\documentclass{amsart}

\usepackage{amsmath, amssymb, amsthm, amscd}       
\newcommand{\Rational}{\mathbb{Q}}
\newcommand{\Real}{\mathbb{R}}

\newcommand{\Adele}{\mathbb{A}}
\newcommand{\GL}{\operatorname{GL}}
\newcommand{\List}[1]{\left\{ #1 \right\}}        
\newcommand{\set}[1]{\mathcal{#1}}                
\newcommand{\Set}[2]{\left\{ #1 \mathrel{\left|\vphantom{#1 \sum #2}\right.} #2 \right\}}          %
\newcommand{\dif}{\,\mathrm{d}}                   
\newcommand{\contra}[1]{\widetilde{#1}}

\theoremstyle{plain}                                                                               %
\newtheorem{theorem}{Theorem}[section]             
\newtheorem*{theorem*}{Theorem}                    
\newtheorem{corollary}[theorem]{Corollary}         
\newtheorem*{corollary*}{Corollary}                
\newtheorem{lemma}[theorem]{Lemma}                 
\newtheorem*{lemma*}{Lemma}                        

\theoremstyle{remark}                                                                              %
\newtheorem*{example*}{Example}                    
\newtheorem{remark}[theorem]{Remark}               
\newtheorem*{remark*}{Remark}                      

\begin{document}

\title[Lower Bounds]{Lower Bounds for Rankin-Selberg $L$-functions on the Edge of the Critical Strip}

\author{Qiao Zhang}
\address{Department of Mathematics \\
Texas Christian University \\
Fort Worth, TX 76129}
\email{q.zhang@tcu.edu}

\date{}

\begin{abstract}
  Let $F$ be a number field, and let $\pi_1$ and $\pi_2$ be distinct unitary cuspidal automorphic representations of $\GL_{n_1}(\Adele_F)$ and $\GL_{n_2}(\Adele_F)$ respectively. In this paper, we derive new lower bounds for the Rankin-Selberg $L$-function $L(s, \pi_1 \times \contra{\pi}_2)$ along the edge $\Re s = 1$ of the critical strip in the $t$-aspect. The corresponding zero-free region for $L(s, \pi_1 \times \contra{\pi}_2)$ is also determined.
\end{abstract}

\subjclass{Primary 11M26; Secondary 11F66}

\keywords{Rankin–Selberg $L$-function, Lower bound, Zero-free region}

\maketitle

\section{Introduction}

The value distributions of $L$-functions on the edge of the critical strip are important in number theory. In particular, the lower bounds for the $L$-values along the edge $\Re s = 1$ are closely related to the determination of zero-free regions of the corresponding $L$-function.

In this paper, we are interested in the lower bounds of Rankin-Selberg $L$-functions along the edge of the critical strip. More precisely, let $F / \Rational$ be a number field of degree $n_F$, and let $\pi_1$ and $\pi_2$ be unitary cuspidal automorphic representations of $\GL_{n_1}(\Adele_F)$ and $\GL_{n_2}(\Adele_F)$ respectively. Assume that the central characters of $\pi_1$ and $\pi_2$ are unitary and normalized so that they are trivial on the diaginally embedded copy of the positive reals. Note that, with this normalization, the Rankin-Selberg $L$-functions $L(s, \pi_1 \times \contra{\pi}_1)$ and $L(s, \pi_2 \times \contra{\pi}_2)$ both have simple poles at $s = 1$. The main subject of this paper is the lower bound for the Rankin-Selberg $L$-function $L(1 + it, \pi_1 \times \contra{\pi}_2)$, which in turn yields a new zero-free region for $L(s, \pi_1 \times \contra{\pi}_2)$.

In 1980, Shahidi~\cite{Shahidi1980} showed that $L(1 + it, \pi_1 \times \contra{\pi}_2) \neq 0$ for every $t \in \Real$. Since then, there has been substantial progress on the determination of lower bounds and zero-free regions for these Rankin Selberg $L$-functions, especially in the following two scenarios. On the one hand, if either $\pi_1$ or $\pi_2$ is self-dual, then it is known, by the work of Moreno~\cite{Moreno1985} in 1985 and Sarnak~\cite{Sarnak2004} in 2004, that $L(s, \pi_1 \times \contra{\pi}_2)$ has the classical zero-free region of the de la Vall\'{e}e-Poussin type, namely
\[
  \sigma > 1 - \frac{c}{\log\Big(\mathfrak{c}(\pi_1) \ \mathfrak{c}(\pi_2) (|t| + 2)\Big)}
\]
where $\mathfrak{c}(\pi_1)$ and $\mathfrak{c}(\pi_2)$ are analytic conductors of $\pi_1$ and $\pi_2$ respectively. On the other hand, the special case with $\pi_1 = \pi_2$ has also attracted much attention. In 2018, Goldfeld and Li~\cite{GoldfeldLi2018} obtained the lower bound
\[
  L(1 + it, \pi \times \contra{\pi}) \gg \frac{1}{(\log (|t| + 2))^3}
\]
where $\pi$ is an irreducible cuspidal unramified representation of $\GL_n(\Adele_{\Rational})$ with $n \geq 2$, tempered almost everywhere. This lower bound immediately leads to the zero-free region for $L(s, \pi \times \contra{\pi})$
\[
  \sigma > 1 - \frac{c_{\pi}}{(\log (|t| + 2))^5}
  \qquad
  |t| \geq 1
\]
These results were improved by Humphries and Brumley~\cite{HumphriesBrumley2019} in 2019 and generalized to an arbitrary number field $F$, with the lower bound
\[
  L(1 + it, \pi \times \contra{\pi}) \gg_{\pi} \frac{1}{\log (|t| + 2)}
\]
and the corresponding zero-free region for $L(s, \pi \times \contra{\pi})$
\begin{equation}\label{eq:Pi-contraPi-ZeroFree-HumphriesBrumley}
  \sigma > 1 - \frac{c_{\pi}}{\log (|t| + 2)}
  \qquad
  |t| \geq 1
\end{equation}
Recently in 2022, Humphries and Thorner~\cite{HumphriesThorner2022} made explicit the dependence of the constant $c_{\pi}$ upon $\pi$, and gave a more precise form for~\eqref{eq:Pi-contraPi-ZeroFree-HumphriesBrumley} given by
\[
  \sigma > 1 - \frac{c}{\log\Big(\mathfrak{c}(\pi)^n (|t| + e)^{n^2n_F}\Big)}
\]
for some effective absolute constant $c > 0$.

The study of general Rankin-Selberg $L$-functions without the self-duality assumption is much more difficult, and our knowledge is rather limited. A lower bound
\[
  L(1 + it, \pi_1 \times \contra{\pi}_2)
  \gg \frac{1}
           {\Big(\mathfrak{c}(\pi_1) \ \mathfrak{c}(\pi_2) (|t| + 2)\Big)^{N_{\pi_1,\pi_2}}}
\]
and the corresponding zero-free region
\[
  \sigma > 1 - \frac{c}
                    {\Big(\mathfrak{c}(\pi_1) \ \mathfrak{c}(\pi_2) (|t|+2)\Big)^{N_{\pi_1,\pi_2}}}
\]
were obtained in 2006 by Brumley~\cite{Brumley2006} in connection with the strong multiplicity one theorem, and by Gelbert and Lapid~\cite{GelbertLapid2006} for a broad family of $L$-functions. In 2013, Brumley~\cite{Brumley2013Appendix} gave an explicit value for $N_{\pi_1,\pi_2}$. More precisely, if $\pi_1 \neq \pi_2$, then he showed in~\cite[Theorem A.1]{Brumley2013Appendix} that
\begin{equation}\label{eq:Brumley2013Appendix-Original}
  L(1 + it, \pi_1 \times \contra{\pi}_2)
  \gg_{F,n_1,n_2,\varepsilon}
    \mathfrak{c}(\Pi_t \times \contra{\Pi}_t)^{-\frac{1}{2}(1-\frac{1}{n_1+n_2})-\varepsilon}
\end{equation}
where (see~\eqref{eq:Conductor-Factorization})
\[
  \mathfrak{c}(\Pi_t \times \contra{\Pi}_t)
  = \mathfrak{c}(\pi_1 \times \contra{\pi}_1) \
    \mathfrak{c}(\pi_2 \times \contra{\pi}_2) \
    \mathfrak{c}(\pi_1 \times \contra{\pi}_2, 1 + it) \
    \mathfrak{c}(\contra{\pi}_1 \times \pi_2, 1 - it)
\]
In particular, by the estimate~\eqref{eq:Bound-Conductor-Pi_t} 
\[
  \mathfrak{c}(\Pi_t \times \contra{\Pi}_t)
  \leq \Big(\mathfrak{c}(\pi_1) \ \mathfrak{c}(\pi_2)\Big)^{2n_1+2n_2}
       (|t| + 2)^{2n_Fn_1n_2}
\]
the lower bound~\eqref{eq:Brumley2013Appendix-Original} can be re-written as
\begin{equation}\label{eq:Brumley2013Appendix-Rewritten}
  L(1 + it, \pi_1 \times \contra{\pi}_2)
  \gg \Big(\mathfrak{c}(\pi_1) \ \mathfrak{c}(\pi_2)\Big)^{-(n_1+n_2-1)-\varepsilon}
      (|t| + 2)^{-n_Fn_1n_2(1-\frac{1}{n_1+n_2})-\varepsilon}
\end{equation}
Further, Brumley~\cite{Brumley2013Appendix} also derived the corresponding zero-free regions for $L(s, \pi_1 \times \contra{\pi}_2)$; in fact, it is shown that these lower bounds still hold in these zero-free regions.

In this paper, we seek to improve upon the above results of Brumley in the $t$-aspect, in particular the lower bound~\eqref{eq:Brumley2013Appendix-Rewritten}, and obtain the following theorem.

\begin{theorem}
  Let $F / \Rational$ be a number field of degree $n_F$, and let $\pi_1$ and $\pi_2$ be unitary cuspidal automorphic representations of $\GL_{n_1}(\Adele_F)$ and $\GL_{n_2}(\Adele_F)$ respectively. Assume that $\pi_1 \neq \pi_2$, and that the central characters of $\pi_1$ and $\pi_2$ are unitary and normalized so that they are trivial on the diagonally embedded copy of the positive reals. Then we have
  \[
    L(\sigma + it, \pi_1 \times \contra{\pi}_2)
    \gg \Big(\mathfrak{c}(\pi_1) \ \mathfrak{c}(\pi_2)\Big)^{-(n_1+n_2-1)-\varepsilon}
        (|t| + 2)^{-\frac{n_Fn_1n_2}{2}\left(1-\frac{1}{n_1+n_2}\right)-\varepsilon}
  \]
  where the $\gg$-constant depends upon $F$, $n_1$, $n_2$ and $\varepsilon$ only.
\end{theorem}

\begin{remark}
  Brumley~\cite{Brumley2013Appendix} used the notation $L(s, \pi_1 \times \contra{\pi}_2)$ to denote the complete Rankin-Selberg $L$-function. Hence our (finite) $L$-function $L(s, \pi_1 \times \contra{\pi}_2)$ is what he denoted by $L^{\infty}(s, \pi_1 \times \contra{\pi}_2)$.
\end{remark}

\begin{remark}
  The work of Brumley~\cite{Brumley2013Appendix} also considered the case $\pi_1 = \pi_2$, and similar results were obtained for the lower bounds and zero-free regions. As it is clear from Section 3, our arguments can also be applied to derive the corresponding improvements of these results in the $t$-aspect. However, these results have now been superseded by the above mentioned works of Goldfeld and Li~\cite{GoldfeldLi2018}, of Humphries and Brumley~\cite{HumphriesBrumley2019}, and of Humphries and Thorne~\cite{HumphriesThorner2022}. Hence we skip the relevant discussions in this paper.
\end{remark}

Applying the convexity bound of Li~\cite{Li2010}, we can also derive from the above lower bound the corresponding zero-free region. Since the arguments are essentially the same as the first part of the proof to~\cite[Theorem A.1]{Brumley2013Appendix}, we omit the details here and only present our result in the following corollary.

\begin{corollary}
  Under the assumptions of Theorem 1, for every $\varepsilon > 0$ there exists a constant $c = c(F, n_1, n_2, \varepsilon) > 0$ such that
  \[
    L(s, \pi_1 \times \contra{\pi}_2)
    \gg_{F,n_1,n_2,\varepsilon}
      \Big(\mathfrak{c}(\pi_1) \ \mathfrak{c}(\pi_2)\Big)^{-(n_1+n_2-1)-\varepsilon}
      (|t| + 2)^{-\frac{d n_1 n_2}{2} \left(1 - \frac{1}{n_1 + n_2}\right)-\varepsilon}
  \]
  in the region
  \[
    \Set{s = \sigma + it}
        {\sigma \geq 1 - \frac{c}
                              {\Big(\mathfrak{c}(\pi_1) \ \mathfrak{c}(\pi_2)\Big)^{(n_1+n_2-1)+\varepsilon}
                               (|t| + 2)^{\frac{dn_1n_2}{2}\left(1-\frac{1}{n_1+n_2}\right)+\varepsilon}}}
  \]
\end{corollary}

\vskip 2ex

\subsection*{Acknowledgements}

The author would like to thank Dorian Goldfeld for stimulating this research, and for many insightful discussions.

\section{Preliminary}

Let $F$ be a number field of degree $n_F$, and let $\pi_1$ and $\pi_2$ be unitary cuspidal automorphic representations of $\GL_{n_1}(\Adele_F)$ and $\GL_{n_2}(\Adele_F)$ respectively. Assume that the central characters of $\pi_1$ and $\pi_2$ are unitary and normalized so that they are trivial on the diagonally embedded copy of the positive reals.

Since the Rankin-Selberg $L$-functions $L(s, \pi_1 \times \contra{\pi}_1)$ and $L(s, \pi_2 \times \contra{\pi}_2)$ both have a simple pole at $s = 1$, we may write their Laurent expansions at $s = 1$ as
\[
  L(s, \pi_1 \times \contra{\pi}_1) = \sum_{k=-1}^{\infty} A_k (s - 1)^k
\]
\[
  L(s, \pi_2 \times \contra{\pi}_2) = \sum_{k=-1}^{\infty} B_k (s - 1)^k
\]

\begin{lemma}\label{lem:Bound-AB}
  For every $\varepsilon > 0$ there exists a constant $C_1' = C_1'(\varepsilon) > 0$ such that
  \[
    |A_{-1}| \leq C_1' \mathfrak{c}(\pi_1)^{\varepsilon}
    \qquad
    |A_0|    \leq C_1' \mathfrak{c}(\pi_1)^{\varepsilon}
  \]
  \[
    |B_{-1}| \leq C_1' \mathfrak{c}(\pi_2)^{\varepsilon}
    \qquad
    |B_0|    \leq C_1' \mathfrak{c}(\pi_2)^{\varepsilon}
  \]
  For simplicity, write
  \begin{equation}\label{eq:Constant-C1}
    C_1 = C_1' \mathfrak{c}(\pi_1)^{\varepsilon} \mathfrak{c}(\pi_2)^{\varepsilon}
  \end{equation}
  then we have
  \[
    |A_{-1}| \leq C_1
    \qquad
    |A_0|    \leq C_1
    \qquad
    |B_{-1}| \leq C_1
    \qquad
    |B_0|    \leq C_1
  \]
\end{lemma}

\begin{proof}
  The lemma follows from Li's convexity bounds~\cite{Li2010}; see also~\cite{Brumley2013Appendix}.
\end{proof}

Further, for every $t \in \Real$ consider the isobaric representation 
\[
  \Pi_t = (\pi_1 \otimes |\det|^{it/2}) \boxplus (\pi_2 \otimes |\det|^{-it/2})
\]
then
we have the factorization of $L$-functions
\begin{equation}\label{eq:Lfunction-Factorization}
  L(s, \Pi_t \times \contra{\Pi}_t)
  = L(s, \pi_1 \times \contra{\pi}_1) \
    L(s, \pi_2 \times \contra{\pi}_2) \
    L(s + it, \pi_1 \times \contra{\pi}_2) \
    L(s - it, \contra{\pi}_1 \times \pi_2)
\end{equation}
and so
the corresponding factorization of analytic conductors
\begin{equation}\label{eq:Conductor-Factorization}
  \mathfrak{c}(\Pi_t \times \contra{\Pi}_t)
  = \mathfrak{c}(\pi_1 \times \contra{\pi}_1) \
    \mathfrak{c}(\pi_2 \times \contra{\pi}_2) \
    \mathfrak{c}(\pi_1 \times \contra{\pi}_2, 1 + it) \
    \mathfrak{c}(\contra{\pi}_1 \times \pi_2, 1 - it)
\end{equation}
In particular, by the bound~\cite[(8)]{Brumley2006} we have
\begin{equation}\label{eq:Bound-Conductor-Pi_t}
  \mathfrak{c}(\Pi_t \times \contra{\Pi}_t)
  \leq \Big(\mathfrak{c}(\pi_1) \ \mathfrak{c}(\pi_2)\Big)^{2n_1+2n_2}
       (|t| + 2)^{2n_Fn_1n_2}
\end{equation}

By the factorization~\eqref{eq:Lfunction-Factorization}, since $\pi_1$ and $\pi_2$ are assumed distinct, the Rankin-Selberg $L$-function $L(s, \Pi_t \times \contra{\Pi}_t)$ has a double pole at $s = 1$ and is holomorphic everywhere else. Accordingly, we write the Laurent series representation of $L(s, \Pi_t \times \contra{\Pi}_t)$ at $s = 1$ as
\[
  L(s, \Pi_t \times \contra{\Pi}_t) = \sum_{k=-2}^{\infty} r_k \ (s - 1)^k
\]

\begin{lemma}\label{lem:Bound-R-Lower}
  For every $\varepsilon > 0$, there exists a constant $C_2' = C_2'(\varepsilon) > 0$ such that
  \begin{align*}
    |r_{-1}| + |r_{-2}|
    \geq& C_2' \ \mathfrak{c}(\Pi_t \times \contra{\Pi}_t)^{-\frac{1}{2}(1-\frac{1}{n_1+n_2})-\varepsilon} \\
    \geq& C_2' \
          \Big(\mathfrak{c}(\pi_1) \ \mathfrak{c}(\pi_2)\Big)^{-(n_1+n_2-1)-\varepsilon}
          (|t| + 2)^{-n_Fn_1n_2(1-\frac{1}{n_1+n_2})-\varepsilon}
  \end{align*}
  For simplicity, write
  \begin{equation}\label{eq:Constant-C2}
    C_2 = C_2' \ \Big(\mathfrak{c}(\pi_1) \ \mathfrak{c}(\pi_2)\Big)^{-(n_1+n_2-1)-\varepsilon}
  \end{equation}
  then we have
  \[
    |r_{-1}| + |r_{-2}|
    \geq C_2 \ (|t| + 2)^{-n_Fn_1n_2(1-\frac{1}{n_1+n_2})-\varepsilon}
  \]
\end{lemma}

\begin{proof}
  This follows immediately from~\cite[Theorem 3]{Brumley2006}.
\end{proof}

\begin{lemma}\label{lem:Bound-R-Upper}
  Let $|t| \geq 1$. Then we have
  \[
    |r_{-1}| + |r_{-2}| \\
    \leq 3 C_1^2 \ \Big|L(1 + it, \pi_1 \times \contra{\pi}_2)\Big|^2
       +   C_1^2 \ \left|\frac{\dif}{\dif t}\Bigg(\Big|L(1 + it, \pi_1 \times \contra{\pi}_2)\Big|^2\Bigg)\right|
  \]
  where $C_1$ is as given in~\eqref{eq:Constant-C1}.
\end{lemma}

\begin{proof}
  By the factorization~\eqref{eq:Lfunction-Factorization}, we have
  \[
    r_{-2} = A_{-1} \ B_{-1} \ \Big|L(1 + it, \pi_1 \times \contra{\pi}_2)\Big|^2
  \]
  so by Lemma~\ref{lem:Bound-AB} we have
  \begin{align*}
    |r_{-2}|
    =&    |A_{-1}| \ |B_{-1}| \ \Big|L(1 + it, \pi_1 \times \contra{\pi}_2)\Big|^2 \\
    \leq& C_1^2 \ \Big|L(1 + it, \pi_1 \times \contra{\pi}_2)\Big|^2
  \end{align*}
  Also, by the factorization~\eqref{eq:Lfunction-Factorization} we have
  \begin{align*}
    r_{-1}
    =& (A_{-1} B_0 + A_0 B_{-1}) \ \Big|L(1 + it, \pi_1 \times \contra{\pi}_2)\Big|^2 \\
     &+ 2 A_{-1} \ B_{-1} \ \Re\Big(L'(1 + it, \pi_1 \times \contra{\pi}_2) \cdot \overline{L(1 + it, \pi_1 \times \contra{\pi}_2)}\Big) \\
    =& (A_{-1} B_0 + A_0 B_{-1}) \ \Big|L(1 + it, \pi_1 \times \contra{\pi}_2)\Big|^2 \\
    &+ A_{-1} \ B_{-1} \ \frac{\dif}{\dif t}\Bigg(\Big|L(1 + it, \pi_1 \times \contra{\pi}_2)\Big|^2\Bigg)
  \end{align*}
  so again by Lemma~\ref{lem:Bound-AB} we have
  \begin{align*}
    |r_{-1}|
    \leq& (|A_{-1}| \ |B_0| + |A_0| \ |B_{-1}|) \
          \Big|L(1 + it, \pi_1 \times \contra{\pi}_2)\Big|^2 \\
       &+ |A_{-1}| \ |B_{-1}| \
          \left|\frac{\dif}{\dif t}\Bigg(\Big|L(1 + it, \pi_1 \times \contra{\pi}_2)\Big|^2\Bigg)\right| \\
    \leq& 2 C_1^2 \ \Big|L(1 + it, \pi_1 \times \contra{\pi}_2)\Big|^2
        + C_1^2 \ \left|\frac{\dif}{\dif t}\Bigg(\Big|L(1 + it, \pi_1 \times \contra{\pi}_2)\Big|^2\Bigg)\right|
  \end{align*}

  Hence our lemma follows as we combine all the above estimates.
\end{proof}

\section{Proof of Theorem 1}

In this section, we present the proof of Theorem 1, and will freely use the notations introduced in Section 2.

  The theorem is well known if $n_1 = n_2 = 1$, so henceforth we assume that either $n_1 \geq 2$ or $n_2 \geq 2$. For simplicity, henceforth we write
  \[
    \theta = \frac{n_F n_1 n_2}{2} \left(1 - \frac{1}{n_1 + n_2}\right)
           \geq \frac{1}{2}
  \]

  Let $\varepsilon > 0$. By~\cite[Theorem A.1]{Brumley2013Appendix}, there exists a constant $C_{\varepsilon}' > 0$ such that
  \[
    \Big|L(1 + it, \pi_1 \times \contra{\pi}_2)\Big|
    \geq C_{\varepsilon}' \
         \mathfrak{c}(\Pi_t \times \contra{\Pi}_t)^{-\frac{1}{2}(1-\frac{1}{n_1+n_2})-\frac{\varepsilon}{2(n_1+n_2+n_Fn_1n_2)}}
  \]
  Recall the bound~\eqref{eq:Bound-Conductor-Pi_t}
  \[
    \mathfrak{c}(\Pi_t \times \contra{\Pi}_t)
    \leq \Big(\mathfrak{c}(\pi_1) \ \mathfrak{c}(\pi_2)\Big)^{2n_1+2n_2}
         (|t| + 2)^{2n_Fn_1n_2}
  \]
  then over the interval $[-1, 1]$ the above lower bound gives
  \begin{align}
    \Big|L(1 + it, \pi_1 \times \contra{\pi}_2)\Big|
    \geq& C_{\varepsilon}' \
          \Big(\mathfrak{c}(\pi_1) \ \mathfrak{c}(\pi_2)\Big)^{-(n_1+n_2-1)-\varepsilon}
          (|t| + 2)^{-2\theta-\varepsilon} \nonumber\\
    \geq& \frac{C_{\varepsilon}'}{3^{\theta}}
          \Big(\mathfrak{c}(\pi_1) \ \mathfrak{c}(\pi_2)\Big)^{-(n_1+n_2-1)-\varepsilon}
          (|t| + 2)^{-\theta-\varepsilon}~\label{tmp:Bound-Interval-pm1}
  \end{align}

  Now we claim that
  \begin{equation}\label{eq:Claim}
    \Big|L(1 + it, \pi_1 \times \contra{\pi}_2)\Big|
    \geq C_{\varepsilon}
         (|t| + 2)^{-\theta-\varepsilon}
    \qquad
    (t \in \Real)
  \end{equation}
  where
  \[
    C_{\varepsilon}
    = \min\List{
        \frac{C_{\varepsilon}'}{3^{\theta}}
          \Big(\mathfrak{c}(\pi_1) \ \mathfrak{c}(\pi_2)\Big)^{-(n_1+n_2-1)-\varepsilon}, \
        \sqrt{\frac{C_2}{6 C_1^2 \ (2 \theta + 2 \varepsilon)}}
      }
    \]
  and $C_1$ and $C_2$ are the constants defined in~\eqref{eq:Constant-C1} and~\eqref{eq:Constant-C2} respectively. Note in particular that we have
  \[
    C_{\varepsilon} \gg \Big(\mathfrak{c}(\pi_1) \ \mathfrak{c}(\pi_2)\Big)^{-(n_1+n_2-1)-\varepsilon}
  \]
  where the $\gg$-constant depends upon $F$, $n_1$, $n_2$ and $\varepsilon$ only.

  By the construction of $C_{\varepsilon}$ and the estimate~\eqref{tmp:Bound-Interval-pm1}, the claim~\eqref{eq:Claim} is obviously valid over the interval $[-1, 1]$. Now assume that there indeed exists some point $|t_0| > 1$ such that
  \begin{equation}\label{tmp:BoundContrary}
    \Big|L(1 + it_0, \pi_1 \times \contra{\pi}_2)\Big| < C_{\varepsilon} \ (|t_0| + 2)^{-\theta-\varepsilon}
  \end{equation}
  Without loss of generality, we may assume that $t_0 > 1$. For simplicity, henceforth we write
  \[
    g(t) = \Big|L(1 + it, \pi_1 \times \contra{\pi}_2)\Big|^2
  \]
  then our claim~\eqref{eq:Claim} becomes
  \[
    g(t) \geq C_{\varepsilon}^2 \ (|t| + 2)^{-2\theta-2\varepsilon}
    \qquad
    (t \in \Real)
  \]
  and the assumption~\eqref{tmp:BoundContrary} becomes
  \begin{equation}\label{tmp:EstimateContradiction}
    g(t_0) < C_{\varepsilon}^2 \ (t_0 + 2)^{-2\theta-2\varepsilon}
  \end{equation}

  Define
  \[
    \set{S}_1 = \Set{0 \leq t < t_0}{g(t) \geq C_{\varepsilon}^2 (t + 2)^{-2\theta-2\varepsilon}}
  \]
  Since $[0, 1] \subseteq \set{S}_1$, we have $\set{S}_1 \neq \emptyset$, so the supremum $t_1 = \sup(\set{S}_1)$ exists and is finite, and we have
  \[
    t_1 \geq 1
    \qquad
    g(t_1) = C_{\varepsilon}^2 \ (t_1 + 2)^{-2\theta-2\varepsilon}
  \]
  Also, we define
  \[
    \set{S}_2 = \Set{t > t_0}{g(t) \geq C_{\varepsilon}^2 \ (t + 2)^{-2\theta-2\varepsilon}}
                \cup
                \List{\infty}
    \qquad
    t_2 = \inf(\set{S}_2)
  \]
  so $t_0 \in (t_1, t_2)$ and we have
  \begin{equation}\label{tmp:Bound-Contrary-t1-t2}
    g(t) < C_{\varepsilon}^2 \ (t + 2)^{-2\theta-2\varepsilon}
    \qquad
    (t_1 < t < t_2)
  \end{equation}
  Further, if $t_2 < \infty$, then we also have
  \[
    g(t_2) = C_{\varepsilon}^2 \ (t_2 + 2)^{-2\theta-2\varepsilon}
  \]

  Combining Lemmas~\ref{lem:Bound-R-Lower} and~\ref{lem:Bound-R-Upper}, for $t \geq 1$ we have
  \[
    C_2 (t + 2)^{-2\theta-\varepsilon}
    \leq |r_{-1}| + |r_{-2}|
    \leq 3 C_1^2 \ g(t) + C_1^2 \ |g'(t)|
  \]
  Hence we have
  \begin{equation}\label{tmp:EstimateDerivative}
    |g'(t)|
    \geq \frac{C_2}{C_1^2} \ (t + 2)^{-2\theta-\varepsilon}
       - 3 g(t)
  \end{equation}
  In particular, over the interval $(t_1, t_2)$, by~\eqref{tmp:Bound-Contrary-t1-t2} we have
  \[
    g(t) < C_{\varepsilon}^2 \ (|t| + 2)^{-2\theta-2\varepsilon}
         \leq \frac{C_2}{6 C_1^2} \ (|t| + 2)^{-2\theta-2\varepsilon}
  \]
  so the estimate~\eqref{tmp:EstimateDerivative} gives
  \begin{align*}
    |g'(t)|
    \geq& \frac{C_2}{C_1^2} \ (t + 2)^{-2\theta-\varepsilon}
        - 3 g(t) \\
    \geq& \frac{C_2}{C_1^2} \ (t + 2)^{-2\theta-\varepsilon}
        - 3 \cdot \frac{C_2}{6 C_1^2} (t + 2)^{-2\theta-2\varepsilon} \\
    \geq& \frac{C_2}{2 C_1^2} \ (t + 2)^{-2\theta-\varepsilon}
    \qquad \qquad \qquad
    (t_1 < t < t_2)
  \end{align*}
  In particular, we have $g'(t) \neq 0$ over $(t_1, t_2)$, so the derivative $g'(t)$ must keep the same sign in this interval. If $g'(t) > 0$ over $(t_1, t_2)$, then $g(t)$ is a strictly increasing function, and in particular
  \[
    g(t_0) > g(t_1)
           = C_{\varepsilon}^2 \ (t_1 + 2)^{-2\theta-2\varepsilon}
           > C_{\varepsilon}^2 \ (t_0 + 2)^{-2\theta-2\varepsilon}
  \]
  in contradiction to the assumption~\eqref{tmp:EstimateContradiction}. Hence we have $g'(t) < 0$ over $(t_1, t_2)$, and so
  \begin{equation}\label{tmp:Bound-Contrary-derivative}
    -g'(t) = |g'(t)| \geq \frac{C_2}{2 C_1^2} \ (t + 2)^{-2\theta-\varepsilon}
    \qquad
    (t_1 < t < t_2)
  \end{equation}

  If $t_2 < \infty$, then by the choices of $t_1$ and $t_2$ we have
  \begin{align*}
    \int_{t_1}^{t_2} \frac{C_2}{2 C_1^2} \ (t + 2)^{-2\theta-\varepsilon} \dif t
    \leq& \int_{t_1}^{t_2} \Big(-g'(t)\Big) \dif t
    =     g(t_1) - g(t_2) \\
    =&    C_{\varepsilon}^2 \ (t_1 + 2)^{-2\theta-2\varepsilon}
     -    C_{\varepsilon}^2 \ (t_2 + 2)^{-2\theta-2\varepsilon} \\
    =&    \int_{t_1}^{t_2} C_{\varepsilon}^2 \ (2\theta + 2 \varepsilon) (t + 2)^{-2\theta-2\varepsilon-1} \dif t
  \end{align*}
  which is simply
  \begin{equation}\label{tmp:Step-2}
    \int_{t_1}^{t_2} C_{\varepsilon}^2 \frac{2\theta + 2 \varepsilon}{(t + 2)^{1+\varepsilon}} (t + 2)^{-2\theta-\varepsilon} \dif t
    \geq
    \int_{t_1}^{t_2} \frac{C_2}{2 C_1^2} (t + 2)^{-2\theta-\varepsilon} \dif t
  \end{equation}
  On the other hand, over the interval $(t_1, t_2)$, by the construction of the constant $C_{\varepsilon}$ we have
  \[
    C_{\varepsilon}^2 \frac{2 \theta + 2 \varepsilon}{(t + 2)^{1+\varepsilon}}
    \leq \frac{C_2}{6 C_1^2 (2 \theta + 2 \varepsilon)} \cdot \frac{2 \theta + 2 \varepsilon}{1 + 2}
    <    \frac{C_2}{2 C_1^2}
  \]
  so we should have
  \[
    \int_{t_1}^{t_2} C_{\varepsilon}^2 \ \frac{2 \theta + 2 \varepsilon}{(t + 2)^{1+\varepsilon}} (t + 2)^{-2\theta-\varepsilon} \dif t
    <
    \int_{t_1}^{t_2} \frac{C_2}{2 C_1^2} \ (t + 2)^{-2\theta-\varepsilon} \dif t
  \]
  a contradiction to~\eqref{tmp:Step-2}. Hence we must have $t_2 = \infty$. In particular, this implies that we can take $t_0$ to be any value in the infinite interval $(t_1, \infty)$.

  In summary, combining~\eqref{tmp:Bound-Contrary-t1-t2} and~\eqref{tmp:Bound-Contrary-derivative} gives
  \[
    g(t) < C_{\varepsilon}^2 \ (t + 2)^{-2\theta-2\varepsilon}
    \qquad
    -g'(t) \geq \frac{C_2}{2 C_1^2} \ (t + 2)^{-2\theta-\varepsilon}
    \qquad
    (t > t_1)
  \]
  so for every $t_0 > t_1$ we have
  \begin{align*}
    C_{\varepsilon}^2 \ (t_1 + 2)^{-2\theta-2\varepsilon}
    =&    g(t_1)
    \geq  g(t_1) - g(t_0)
    =     \int_{t_1}^{t_0} \Big(-g'(t)\Big) \dif t \\
    \geq& \int_{t_1}^{t_0} \frac{C_2}{2 C_1^2} \ (t + 2)^{-2\theta-\varepsilon} \dif t
    =     \frac{C_2}{2 C_1^2} \left.\frac{(t + 2)^{1-2\theta-\varepsilon}}{1 - 2 \theta - \varepsilon}\right|_{t_1}^{t_0} \\
    =&    \frac{C_2}{2 C_1^2} \
          \frac{(t_1 + 2)^{1-2\theta-\varepsilon} - (t_0 + 2)^{1-2\theta-\varepsilon}}
               {2 \theta + \varepsilon - 1}
  \end{align*}
  Since $\theta \geq 1 / 2$, we have $1 - 2 \theta - \varepsilon > 0$, so letting $t_0 \to \infty$ gives
  \[
    C_{\varepsilon}^2 (t_1 + 2)^{-2\theta-2\varepsilon}
    \geq \frac{C_2}{2 C_1^2} \ \frac{(t_1 + 2)^{1-2\theta-\varepsilon}}
              {2\theta + \varepsilon - 1}
  \]
  This implies that
  \[
    C_{\varepsilon}^2 \geq \frac{C_2}{2 C_1^2} \
                           \frac{(t_1 + 2)^{1+\varepsilon}}
                                {2 \theta + \varepsilon - 1}
  \]
  and so
  \begin{align*}
    (t_1 + 2)^{1+\varepsilon}
    \leq& C_{\varepsilon}^2 \cdot \frac{2 C_1^2}{C_2}
          \cdot
          (2 \theta + \varepsilon - 1) \\
    \leq& \frac{C_2}{6 C_1^2 (2 \theta + 2 \varepsilon)}
          \cdot \frac{2 C_1^2}{C_2}
          \cdot (2 \theta + \varepsilon - 1) \\
    =&    \frac{2 \theta + \varepsilon - 1}{3 (\theta + \varepsilon)}
    <     1
  \end{align*}
  which is impossible. Hence there cannot exist a point $|t_0| > 1$ satisfying~\eqref{tmp:EstimateContradiction}, and so the claim~\eqref{eq:Claim} is valid for every $t \in \Real$, i.e.,
  \[
    \Big|L(1 + it, \pi_1 \times \contra{\pi}_2)\Big| \geq C_{\varepsilon} \ (|t| + 2)^{-\theta-\varepsilon}
    \qquad
    (t \in \Real)
    \qedhere
  \]

This completes our proof of Theorem 1.

\end{document}